\documentclass[10]{article}
\usepackage{amsfonts,amssymb,amsmath,amsthm,cite}
\usepackage{graphicx}

\usepackage[applemac]{inputenc}
\usepackage{amsmath,amssymb,amsthm, hyperref, euscript}
\usepackage[matrix,arrow,curve]{xy}
\usepackage{graphicx}
\usepackage{tabularx}
\usepackage{float}
\usepackage{hyperref}
\usepackage{tikz}
\usepackage{slashed}
\usepackage{mathrsfs}

\usetikzlibrary{matrix}

\usepackage[T1]{fontenc}
\usepackage{amsfonts,cite}
\usepackage{graphicx}

\usepackage[applemac]{inputenc}

\usepackage[sc]{mathpazo}
\DeclareFontFamily{T1}{pzc}{}
\DeclareFontShape{T1}{pzc}{m}{it}{1.8 <-> pzcmi8t}{}
\DeclareMathAlphabet{\mathpzc}{T1}{pzc}{m}{it}
\theoremstyle{plain}
\newtheorem{prop}{Proposition}[section]

\newtheorem{lem}[prop]{Lemma}%[section]
\newtheorem{thm}[prop]{Theorem}%[section]

\theoremstyle{definition}
\newtheorem{defn}[prop]{Definition}%[section]
\newtheorem{empt}[prop]{}%[section]
\newtheorem{exm}[prop]{Example}%[section]
\newtheorem{rem}[prop]{Remark}%[section]

\DeclareMathOperator{\Dom}{Dom}              %% domain of an operator
\newcommand{\vertiii}[1]{{\left\vert\kern-0.25ex\left\vert\kern-0.25ex\left\vert #1
    \right\vert\kern-0.25ex\right\vert\kern-0.25ex\right\vert}}
\usepackage[sc]{mathpazo}
\newbox\ncintdbox \newbox\ncinttbox %% noncommutative integral symbols
\setbox0=\hbox{$-$} \setbox2=\hbox{$\displaystyle\int$}
\setbox\ncintdbox=\hbox{\rlap{\hbox
    to \wd2{\hskip-.125em \box2\relax\hfil}}\box0\kern.1em}
\setbox0=\hbox{$\vcenter{\hrule width 4pt}$}
\setbox2=\hbox{$\textstyle\int$} \setbox\ncinttbox=\hbox{\rlap{\hbox
    to \wd2{\hskip-.175em \box2\relax\hfil}}\box0\kern.1em}

\title{Noncommutative Local Systems}
\begin{document}
\maketitle  \setlength{\parindent}{0pt}
\begin{center}
\author{
{\textbf{Petr R. Ivankov*}\\
e-mail: * monster.ivankov@gmail.com \\
}
}
\end{center}

\vspace{1 in}

%\begin{abstract}
\noindent

\paragraph{}	Gelfand - Na\u{i}mark theorem supplies a one to one correspondence between commutative $C^*$-algebras and locally compact Hausdorff spaces. So any noncommutative $C^*$-algebra can be regarded as a generalization of a topological space.  Generalizations of several topological invariants may be defined by algebraic methods. For example Serre Swan theorem \cite{karoubi:k} states that complex topological $K$-theory coincides with $K$-theory of $C^*$-algebras. This article is concerned with generalization of local systems. The classical construction of local system implies an existence of a path groupoid. However the noncommutative geometry does not contain this object. There is a construction of local system which uses covering projections. Otherwise a classical (commutative) notion of a covering projection has a noncommutative generalization. A generalization of noncommutative covering projections supplies a generalization of local systems.
%\end{abstract}
\tableofcontents

\section{Motivation. Preliminaries}

\paragraph{}  Local system examples  arise geometrically from vector bundles with flat connections, and from topology by means of linear representations of the fundamental group. Generalization of local systems requires a generalization of a topological space given by the Gelfand-Na\u{i}mark theorem \cite{arveson:c_alg_invt} which states the correspondence between  locally compact Hausdorff topological spaces and commutative $C^*$-algebras.

\begin{thm}\label{gelfand-naimark}\cite{arveson:c_alg_invt}
Let $A$ be a commutative $C^*$-algebra and let $\mathcal{X}$ be the spectrum of A. There is the natural $*$-isomorphism $\gamma:A \to C_0(\mathcal{X})$.
\end{thm}

\paragraph{} So any (noncommutative) $C^*$-algebra may be regarded as a generalized (noncommutative)  locally compact Hausdorff topological space. We would like to generalize a notion of a local system. A classical notion of local system uses a fundamental groupoid.
\begin{thm}
\cite{spanier:at} For each topological space there is a category $\mathscr{P}(\mathcal{X})$ whose objects are points of $\mathcal{X}$, whose morphisms from $x_0$ to $x_1$ are the path classes with $x_0$ as origin and $x_1$ as end, and whose composite is the product of path classes.
\end{thm}
\begin{defn}
\cite{spanier:at} The category $\mathscr{P}(\mathcal{X})$ is called the {\it category of path classes} of $\mathcal{X}$ or the {\it fundamental groupoid}.
\end{defn}
\begin{defn}
\cite{spanier:at} A {\it local system} on a space $\mathcal{X}$ is a covariant functor from fundamental groupoid of $\mathcal{X}$ to some category. For any category $\mathscr{C}$ there is a category of local systems on $\mathcal{X}$ with values in $\mathscr{C}$. Two local systems are said to be {\it equivalent} if they are equivalent objects in this category.
\end{defn}
\paragraph{} Otherwise it is known that any connected gruopoid is equivalent to a category with single object, i.e. a groupoid is equivalent to a group which is regarded as a category. Any groupoid can be decomposed into connected components, therefore any local system corresponds to representations of groups. It means that in case of linearly connected space $\mathcal{X}$ local systems can be defined by representations of fundamental group $\pi_1(\mathcal{X})$. Otherwise there is an interrelationship between fundamental group and covering projections. This circumstance supplies a following definition \ref{borel_const_comm} and a lemma \ref{borel_local_system_app}  which do not explicitly uses a fundamental groupoid. \begin{defn}\label{borel_const_comm}\cite{davis_kirk_at}
Let $p : \mathcal{P} \to \mathcal{B}$ be a principal $G$-bundle. Suppose $G$ acts
on the left on a space $\mathcal{F}$, i.e. an action $G \times \mathcal{F} \to \mathcal{F}$ is given. Define the
{\it Borel construction}
\begin{equation*}
\mathcal{P} \times_G \mathcal{F}
\end{equation*}

to be the quotient space $\mathcal{P} \times \mathcal{F} / \approx$  where
\begin{equation*}
\left(x, f\right) \approx \left(xg, g^{-1}f\right).
\end{equation*}

\end{defn}
We next give one application of the Borel construction. Recall that a
local coefficient system is a fiber bundle over $B$ with fiber $A$ and structure
group $G$ where $A$ is a (discrete) abelian group and G acts via a homomorphism
$G \to \mathrm{Aut}(A)$.
\begin{lem}\label{borel_local_system_app}\cite{davis_kirk_at}
Every local coefficient system over a path-connected (and semilocally
simply connected) space $B$ is of the form

\begin{tikzpicture}\label{borel_local_comm}
  \matrix (m) [matrix of math nodes,row sep=3em,column sep=4em,minimum width=2em]
  {
     A &  \widetilde{\mathcal{B}} \times_{\pi_1(\mathcal{B})}A \\
      & B \\};
  \path[-stealth]
    (m-1-1) edge node [left] {} (m-1-2)
    (m-1-2) edge node [right] {$q$} (m-2-2);

\end{tikzpicture}

i.e., is associated to the principal $\pi_1(\mathcal{B})$-bundle given by the universal cover $\widetilde{\mathcal{B}}$
of $\mathcal{B}$ where the action is given by a homomorphism $\pi_1(\mathcal{B}) \to \mathrm{Aut}(A)$.
\end{lem}

In lemma \ref{borel_local_comm} the $\mathcal{B}$ is a topological space, the $\widetilde{\mathcal{B}}$ means the universal covering space of  $\mathcal{B}$, $\pi = \pi_1(\mathcal{B})$ is the fundamental group of the  $\mathcal{B}$. The $\pi$ group equals to group of covering transformations $G(\widetilde{\mathcal{B}}, \mathcal{B})$ of the universal covering space. So above construction does not need fundamental groupoid, it uses a covering projection and a group of covering transformations. 
However noncommutative generalizations of these notions are developed in \cite{ivankov:infinite_cov_pr}. So local systems can be generalized. We may summarize several properties of the Gelfand - Na\u{i}mark correspondence with the
following dictionary.
\newline
\break
\begin{tabular}{|c|c|}
\hline
TOPOLOGY & ALGEBRA\\
\hline
Locally compact space & $C^*$-algebra\\
Covering projection  & Noncommutative covering projection \\
Group of covering transformations  & Noncommutative group of covering transformations \\
Local system  & ? \\
\hline
\end{tabular}
\newline
\newline
\break

This article assumes elementary knowledge of following subjects:
\begin{enumerate}
\item Set theory \cite{halmos:set},
\item Category theory  \cite{spanier:at},
\item Algebraic topology  \cite{spanier:at},
\item $C^*$-algebras and operator theory \cite{pedersen:ca_aut},
\item Differential geometry \cite{koba_nomi:fgd},
\item Spectral triples and their connections \cite{connes:c_alg_dg,connes:ncg94,varilly:noncom,varilly_bondia}. 
\end{enumerate}

The terms "set", "family" and "collection" are synonyms.
Following table contains used in this paper notations.
\newline
\begin{tabular}{|c|c|}
\hline
Symbol & Meaning\\
\hline
\\
%$A^+$  & Unitization of $C^*-$ algebra $A$\\
$A^G$  & Algebra of $G$ invariants, i.e. $A^G = \{a\in A \ | \ ga=a, \forall g\in G\}$\\
$\mathrm{Aut}(A)$ & Group * - automorphisms of $C^*$  algebra $A$\\
$B(H)$ & Algebra of bounded operators on Hilbert space $H$\\
%$B_{\infty}=B_{\infty}(\{z\in \mathbb{C} \ | \ |z|=1\})$  & Algebra of Borel measured functions on the $\{z\in \mathbb{C} \ | \ |z|=1\}$ set. \\
$\mathbb{C}$ (resp. $\mathbb{R}$)  & Field of complex (resp. real) numbers \\
%$\mathbb{C}^*$ & $\{z \in \mathbb{C} \ | \ |z| = 1\}$ \\
$C(\mathcal{X})$ & $C^*$ - algebra of continuous complex valued \\
 & functions on topological space $\mathcal{X}$\\
$C_0(\mathcal{X})$ & $C^*$ - algebra of continuous complex valued \\
 & functions on topological space $\mathcal{X}$\\
% $C_b(\mathcal{X})$ & $C^*$ - algebra of bounded  continuous complex valued \\
  & functions on topological space $\mathcal{X}$\\
  %$G_{tors} \subset G$  & The torsion subgroup of an abelian group\\
$G(\widetilde{\mathcal{X}} | \mathcal{X})$ & Group of covering transformations of covering projection  \\
 & $\widetilde{\mathcal{X}} \to \mathcal{X}$ \cite{spanier:at}\\
$H$ &Hilbert space \\
%$I = [0, 1] \subset \mathbb{R}$ & Closed unit  interval\\
%$K(A)$ & Pedersen ideal of $C^*$-algebra $A$\\
%$\mathcal{K}(H)$ or $\mathcal{K}$ & Algebra of compact operators on Hilbert space $H$\\
%$\mathbb{M}_n(A)$  & The $n \times n$ matrix algebra over $C^*-$ algebra $A$\\
$M(A)$  & A multiplier algebra of $C^*$-algebra $A$\\

$\mathscr{P}(\mathcal{X})$  & Fundamental groupoid of a topological space $\mathcal{X}$\\

%$\mathbb{Q}$  & Field of rational numbers \\
 % $\mathrm{sp}(a)$ & Spectrum of element of $C^*$-algebra $a\in A$  \\
%  $\mathrm{supp}(f)$ & Support of $f\in C_0(\mathcal{X})$, $\mathrm{supp}(f) = \left\{x \in \mathcal{X} \ | \ f(x)\neq 0 \right\}$   \\
$U(H) \subset \mathcal{B}(H) $ & Group of unitary operators on Hilbert space $H$\\
$U(A) \subset A $ & Group of unitary operators of algebra $A$\\
$U(n) \subset GL(n, \mathbb{C}) $ & Unitary subgroup of general linear group\\

$\mathbb{Z}$ & Ring of integers \\

$\mathbb{Z}_m$ & Ring of integers modulo $m$ \\
$\Omega$ &  Natural contravariant functor from category  of commutative \\ & $C^*$ - algebras, to category of Hausdorff spaces\\

\hline
\end{tabular}

\section{Noncommutative covering projections}
\paragraph{}
In this section we recall the described in \cite{ivankov:infinite_cov_pr} construction of a noncommutative covering projection. Instead the expired "rigged space" notion  we use the "Hilbert module" one.

\subsection{Hermitian modules and functors}

 \begin{defn}
\cite{rieffel_morita} Let $B$ be a $C^*$-algebra. By a (left) {\it Hermitian $B$-module} we will mean the Hilbert space $H$ of a non-degenerate *-representation $A \rightarrow B(H)$. Denote by $\mathbf{Herm}(B)$ the category of Hermitian $B$-modules.
 \end{defn}
\paragraph{}
 Let $A$, $B$ be $C^*$-algebras. In this section we will study some general methods for construction of functors from  $\mathbf{Herm}(B)$ to  $\mathbf{Herm}(A)$.

\begin{defn} \cite{rieffel_morita}
Let $B$ be a $C^*$-algebra. By (right) {\it pre-$B$-Hilbert module} we mean a vector space, $X$, over complex numbers on which $B$ acts by means of linear transformations in such a way that $X$ is a right $B$-module (in algebraic sense), and on which there is defined a $B$-valued sesquilinear form $\langle,\rangle_X$ conjugate linear in the first variable, such that
\begin{enumerate}
\item $\langle x, x \rangle_B \ge 0$
\item $\left(\langle x, y \rangle_X\right)^* = \langle y, x \rangle_X$
\item $\langle x, yb \rangle_X = \langle x, y \rangle_Xb$.
\end{enumerate}
\end{defn}
\begin{empt}
It is easily seen that if we factor a pre-$B$-Hilbert module by subspace of the elements $x$ for which $\langle x, x \rangle_X = 0$, the quotient becomes in a natural way a pre-$B$-Hilbert module having the additional property that inner product is definite, i.e. $\langle x, x \rangle_X > 0$ for any non-zero $x\in X$. On a pre-$B$-Hilbert module with definite inner product we can define a norm $\|\|$ by setting
\begin{equation}\label{rigged_norm_eqn}
\|x\|=\|\langle x, x \rangle_X\|^{1/2}.
\end{equation}
From now we will always view a  pre-$B$-Hilbert module  with definite inner product as being equipped with this norm. The completion of $X$ with this norm is easily seen to become again a pre-$B$-Hilbert module.
\end{empt}
\begin{defn}
\cite{rieffel_morita} Let $B$ be a $C^*$-algebra. By a {\it Hilbert $B$-module} we will mean a pre-$B$-Hilbert module, $X$, satisfying the following conditions:
\begin{enumerate}
\item If $\langle x, x \rangle_X\ = 0$ then $x = 0$, for all $x \in X$
\item $X$ is complete for the norm defined in (\ref{rigged_norm_eqn}).
\end{enumerate}
\end{defn}
\begin{exm}\label{fin_rigged_exm}
Let $A$ be a $C^*$-algebra and a finite group acts on $A$, $A^G$ is the algebra of $G$-invariants. Then $A$ is a Hilbert $A^G$-module on which is defined following $A^G$-valued form
 \begin{equation}\label{inv_scalar_product}
 \langle x, y \rangle_A = \frac{1}{|G|} \sum_{g \in G} g(x^*y).
 \end{equation}
 Since given by \ref{inv_scalar_product} sum is $G$-invariant we have $ \langle x, y \rangle_A \in A^G$.
\end{exm}
\paragraph{}
Viewing a Hilbert $B$-module as a generalization of an ordinary Hilbert space, we can define what we mean by bounded operators on a Hilbert $B$-module.

\begin{defn}\cite{rieffel_morita}
Let $X$ be a Hilbert $B$-module. By a {\it bounded operator} on $X$ we mean a linear operator, $T$, from $X$ to itself which satisfies following conditions:
\begin{enumerate}
\item for some constant $k_T$ we have
\begin{equation}\nonumber
\langle Tx, Tx \rangle_X \le k_T \langle x, x \rangle_X, \ \forall x\in X,
\end{equation}
or, equivalently $T$ is continuous with respect to the norm of $X$.
\item there is a continuous linear operator, $T^*$, on $X$ such that
\begin{equation}\nonumber
\langle Tx, y \rangle_X = \langle x, T^*y \rangle_X, \ \forall x, y\in X.
\end{equation}
\end{enumerate}
It is easily seen that any bounded operator on a $B$-Hilbert module will automatically commute with the action of $B$ on $X$ (because it has an adjoint). We will denote by $\mathcal{L}(X)$ (or $\mathcal{L}_B(X)$ there is a chance of confusion) the set of all bounded operators on $X$. Then it is easily verified than with the operator norm $\mathcal{L}(X)$ is a $C^*$-algebra.
\end{defn}
\begin{defn}\cite{pedersen:ca_aut} If $X$ is a Hilbert $B$-module then
denote by $\theta_{\xi, \zeta} \in \mathcal{L}_B(X)$   such that
\begin{equation}\nonumber
\theta_{\xi, \zeta} (\eta) = \zeta \langle\xi, \eta \rangle_X , \ (\xi, \eta, \zeta \in X)
\end{equation}
Norm closure of  a generated by such endomorphisms ideal is said to be the {\it algebra of compact operators} which we denote by $\mathcal{K}(X)$. The $\mathcal{K}(X)$ is an ideal of  $\mathcal{L}_B(X)$. Also we shall use following notation $\xi\rangle \langle \zeta \stackrel{\text{def}}{=} \theta_{\xi, \zeta}$.
\end{defn}

\begin{defn}\cite{rieffel_morita}\label{corr_defn}
Let $A$ and $B$ be $C^*$-algebras. By a {\it Hilbert $B$-$A$-correspondence} we mean a Hilbert $B$-module, which is a left $A$-module by means of *-homomorphism of $A$ into $\mathcal{L}_B(X)$.
\end{defn}

\begin{empt}\label{herm_functor_defn}
Let $X$ be a Hilbert $B$-$A$-correspondence. If $V\in \mathbf{Herm}(B)$ then we can form the algebraic tensor product $X \otimes_{B_{\mathrm{alg}}} V$, and equip it with an ordinary pre-inner-product which is defined on elementary tensors by
\begin{equation}\nonumber
\langle x \otimes v, x' \otimes v' \rangle = \langle \langle x',x \rangle_B v, v' \rangle_V.
\end{equation}
Completing the quotient $X \otimes_{B_{\mathrm{alg}}} V$ by subspace of vectors of length zero, we obtain an ordinary Hilbert space, on which $A$ acts (by $a(x \otimes v)=ax\otimes v$) to give a  *-representation of $A$. We will denote the corresponding Hermitian module by $X \otimes_{B} V$. The above construction defines a functor $X \otimes_{B} -: \mathbf{Herm}(B)\to \mathbf{Herm}(A)$ if for $V,W \in \mathbf{Herm}(B)$ and $f\in \mathrm{Hom}_B(V,W)$ we define $f\otimes X \in \mathrm{Hom}_A(V\otimes X, W\otimes X)$ on elementary tensors by $(f \otimes X)(x \otimes v)=x \otimes f(v)$.	We can define action of $B$ on $V\otimes X$ which is defined on elementary tensors by
\begin{equation}\nonumber
b(x \otimes v)= (x \otimes bv) = x b \otimes v.
\end{equation}
\end{empt}

\subsection{Galois correspondences}
\begin{defn}\label{herm_a_g_defn}

Let $A$ be a $C^*$-algebra, $G$ is a finite or countable group which acts on $A$. We say that  $H \in \mathbf{Herm}(A)$ is a {\it $A$-$G$ Hermitian module} if
\begin{enumerate}
\item Group $G$ acts on $H$ by unitary $A$-linear isomorphisms,
\item There is a subspace $H^G \subset H$ such that
\begin{equation}\label{g_act}
H = \bigoplus_{g\in G}gH^G.
\end{equation}
\end{enumerate}
Let $H$, $K$ be  $A$-$G$ Hermitian modules, a morphism $\phi: H\to K$ is said to be a $A$-$G$-morphism if $\phi(gx)=g\phi(x)$ for any $g \in G$. Denote by $\mathbf{Herm}(A)^G$ a category of  $A$-$G$ Hermitian modules and $A$-$G$-morphisms.
\end{defn}
\begin{rem}
Condition 2 in the above definition is introduced because any topological covering projection $\widetilde{\mathcal{X}} \to \mathcal{X}$ commutative $C^*$ algebras $C_0\left(\widetilde{\mathcal{X}}\right)$, $C_0\left(\mathcal{X}\right)$ satisfies it with respect to the group of covering transformations $G(\widetilde{\mathcal{X}}| \mathcal{X})$. 
\end{rem}

\begin{defn}
Let $H$ be $A$-$G$ Hermitian module, $B\subset M(A)$ is sub-$C^*$-algebra such that $(ga)b = g(ab)$,  $b(ga) = g(ba)$, for any $a\in A$, $b \in B$, $g \in G$.
There is a functor $(-)^G: \mathbf{Herm}(A)^G \to\mathbf{Herm}(B)$ defined by following way
\begin{equation}
H \mapsto H^G.
\end{equation}
This functor is said to be the {\it invariant functor}.
\end{defn}

\begin{defn}
Let $_AX_B$ be a Hilbert $B$-$A$ correspondence, $G$ is finite or countable group such that
\begin{itemize}
\item $G$ acts on $A$ and $X$,
\item Action of $G$ is equivariant, i.e. $g (a\xi) = (ga) (g\xi)$ , and $B$ invariant, i.e. $g(\xi b)=(g\xi)b$ for any $\xi \in X$, $b \in B$, $a\in A$, $g \in G$,
\item Inner-product  of $G$ is equivariant, i.e. $\langle g\xi, g \zeta\rangle_X = \langle\xi, \zeta\rangle_X$ for any $\xi, \zeta \in X$,  $g \in G$.
\end{itemize}
Then we say that  $_AX_B$ is a {\it $G$-equivariant Hilbert $B$-$A$-correspondence}.
\end{defn}
\paragraph{}
Let $_AX_B$ be a  $G$-equivariant Hilbert $B$-$A$-correspondence. Then for any  $H\in \mathbf{Herm}(B)$ there is an action of $G$  on $X\otimes_B H$ such that
\begin{equation*}
g \left(x \otimes \xi\right) = \left(x \otimes g\xi\right).
\end{equation*}

\begin{defn}\label{inf_galois_defn}
Let $_AX_B$ be a  $G$-equivariant Hilbert $B$-$A$-correspondence. We say that  $_AX_B$ is {\it $G$-Galois Hilbert $B$-$A$-correspondence} if it satisfies following conditions:
\begin{enumerate}
\item  $X \otimes_B H$ is a $A$-$G$ Hermitian module, for any $H \in \mathbf{Herm}(B)$,
\item A pair   $\left(X \otimes_B -, \left(-\right)^G\right)$ such that
\begin{equation}\nonumber
X \otimes_B -: \mathbf{Herm}(B) \to \mathbf{Herm}(A)^G,
\end{equation}
\begin{equation}\nonumber
(-)^G: \mathbf{Herm}(A)^G \to \mathbf{Herm}(B).
\end{equation}

is a pair of inverse equivalence.

\end{enumerate}

\end{defn}
Following theorem is an analog of to theorems described in \cite{miyashita_infin_outer_gal}, \cite{takeuchi:inf_out_cov}.
\begin{thm}\cite{ivankov:infinite_cov_pr}\label{main_lem}
Let $A$ and $\widetilde{A}$ be $C^*$-algebras,  $_{\widetilde{A}}X_A$ be a $G$-equivariant Hilbert $A$-$\widetilde{A}$-correspondence. Let $I$ be a finite or countable set of indices,  $\{e_i\}_{i\in I} \subset M(A)$, $\{\xi_i\}_{i\in I} \subset \ _{\widetilde{A}}X_A$ such that
\begin{enumerate}
\item
\begin{equation}\label{1_mb}
1_{M(A)} =  \sum_{i\in I}^{}e^*_ie_i,
\end{equation}
\item
\begin{equation}\label{1_mkx}
1_{M(\mathcal{K}(X))} = \sum_{g\in G}^{} \sum_{i \in I}^{}g\xi_i\rangle \langle g\xi_i ,
\end{equation}
\item
\begin{equation}\label{ee_xx}
\langle \xi_i, \xi_i \rangle_X = e_i^*e_i,
\end{equation}
\item
\begin{equation}\label{g_ort}
\langle g\xi_i, \xi_i\rangle_X=0, \ \text{for any nontrivial} \ g \in G.
\end{equation}
\end{enumerate}
Then $_{\widetilde{A}}X_A$ is a $G$-Galois Hilbert $A$-$\widetilde{A}$-correspondence.

\end{thm}

\begin{defn}
Consider a situation from the theorem \ref{main_lem}. Let us consider two specific cases
\begin{enumerate}
\item $e_i \in A$ for any $i \in I$,
\item $\exists i \in I \ e_i \notin A$.
\end{enumerate}

Norm completion of the generated by operators
\begin{equation*}
g\xi_i^* \rangle \langle g \xi_i \ a; \ g \in G, \ i \in I, \ \begin{cases}
   a \in M(A), & \text{in case 1},\\
    a \in A, & \text{in case 2}.
  \end{cases}
\end{equation*}
algebra is said to be the {\it subordinated to $\{\xi_i\}_{i \in I}$ algebra}. If $\widetilde{A}$ is the subordinated to $\{\xi_i\}_{i \in I}$ then
\begin{enumerate}
\item $G$ acts on $\widetilde{A}$ by following way
\begin{equation*}
g \left( \ g'\xi_i \rangle \langle g' \xi_i \ a \right) =  gg'\xi_i \rangle \langle gg' \xi_i \ a; \ a \in M(A).
\end{equation*}
\item $X$ is a left $A$ module, moreover $_{\widetilde{A}}X_A$ is a  $G$-Galois Hilbert $A$-$\widetilde{A}$-correspondence.
\item There is a natural $G$-equivariant *-homomorphism $\varphi: A \to M\left(\widetilde{A}\right)$, $\varphi$ is equivariant, i.e.
\begin{equation}
 \varphi(a)(g\widetilde{a})= g \varphi(a)(\widetilde{a}); \ a \in A, \ \widetilde{a}\in \widetilde{A}.
\end{equation}
\end{enumerate}
A quadruple $\left(A, \widetilde{A}, _{\widetilde{A}}X_A, G\right)$ is said to be a {\it Galois quadruple}. The group $G$ is said to be a {\it group Galois transformations} which shall be denoted by $G\left(\widetilde{A}\ | \ A\right)=G$.
\end{defn}
\begin{rem}
Henceforth subordinated algebras only are regarded as noncommutative generalizations of covering projections.
\end{rem}
\begin{defn}
If $G$ is finite then bimodule $_{\widetilde{A}}X_A$ can be replaced with $_{\widetilde{A}}\widetilde{A}_A$ where product $\langle \ , \ \rangle_{\widetilde{A}}$ is given by \eqref{inv_scalar_product}. In this case a Galois quadruple $\left(A, \widetilde{A}, _{\widetilde{A}}X_A, G\right)=\left(A, \widetilde{A}, _{\widetilde{A}}A_A, G\right)$ can be replaced with a {\it Galois triple}  $\left(A, \widetilde{A}, G\right)$.
\end{defn}

\subsection{Infinite noncommutative covering projections}
\paragraph{} In case of commutative $C^*$-algebras definition \ref{inf_galois_defn} supplies algebraic formulation of infinite covering projections of topological spaces. However I think that above definition is not a quite good analogue of noncommutative covering projections. Noncommutative algebras contains inner automorphisms. Inner automorphisms are rather gauge transformations \cite{gross_gauge} than geometrical ones. So I think that inner automorphisms should be excluded. Importance of outer automorphisms was noted by  Miyashita \cite{miyashita_fin_outer_gal,miyashita_infin_outer_gal}. It is reasonably take to account outer automorphisms only. I have set more strong condition.
\begin{defn}\label{gen_in_def}\cite{rieffel_finite_g}
Let  $A$ be $C^*$-algebra. A *-automorphism $\alpha$ is said to be {\it generalized inner} if it is given by conjugating with unitaries from multiplier algebra $M(A)$.
\end{defn}
\begin{defn}\label{part_in_def}\cite{rieffel_finite_g}
Let  $A$ be $C^*$ - algebra. A *- automorphism $\alpha$ is said to be {\it partly inner} if its restriction to some non-zero $\alpha$-invariant two-sided ideal is generalized inner. We call automorphism {\it purely outer} if it is not partly inner.
\end{defn}
Instead definitions \ref{gen_in_def}, \ref{part_in_def} following definitions are being used.
\begin{defn}
Let $\alpha \in \mathrm{Aut}(A)$ be an automorphism. A representation $\rho : A\rightarrow B(H)$ is said to be {\it $\alpha$ - invariant} if a representation $\rho_{\alpha}$ given by
\begin{equation*}
\rho_{\alpha}(a)= \rho(\alpha(a))
\end{equation*}
is unitary equivalent to $\rho$.
\end{defn}
\begin{defn}
Automorphism $\alpha \in \mathrm{Aut}(A)$ is said to be {\it strictly outer} if for any $\alpha$- invariant representation $\rho: A \rightarrow B(H) $, automorphism $\rho_{\alpha}$ is not a generalized inner automorphism.
\end{defn}
\begin{defn}\label{nc_fin_cov_pr_defn}
A Galois quadruple  $\left(A, \widetilde{A}, _{\widetilde{A}}X_A, G\right)$ (resp. a triple $\left(A, \widetilde{A}, G\right)$) with countable (resp. finite) $G$ is said to be a {\it noncommutative infinite (resp. finite)
covering projection} if action of $G$ on $\widetilde{A}$ is strictly outer.
\end{defn}

\section{Noncommutative generalization of local systems}

\begin{defn}\label{loc_sys_defn}
Let $A$ be a $C^*$-algebra, and let $\mathscr{C}$ be a category. A {\it noncommutative local system} contains following ingredients:
\begin{enumerate}
\item A noncommutative covering projection $\left(A, \widetilde{A}, _{\widetilde{A}}X_A, G\right)$ (or $\left(A, \widetilde{A}, G\right)$),
\item A covariant functor $F: G \to \mathscr{C}$,
\end{enumerate}
where $G$ is regarded as a category with a single object $e$, which is the unity of $G$. Indeed a local system is a group homomorphism $G \to \mathrm{Aut}(F(e))$. 
\end{defn}

\begin{exm}
If $\mathcal{X}$ is a linearly connected  space then there is the equivalence of categories $\mathscr{P}(\mathcal{X}) \approx \pi_1(\mathcal{X})$. Let $F: \mathscr{P}(\mathcal{X})\to \mathscr{C}$ is a local system then there is an object $A$ in $\mathscr{C}$ such that $F$ is uniquely defined by a group homomorphism $f: \pi_1(\mathcal{X}) \to \mathrm{Aut}(A)$. Let $G=\pi_1(\mathcal{X})/\mathrm{ker}f$ be a factor group and let $\widetilde{\mathcal{X}} \to \mathcal{X}$ be a covering projection such that $G(\widetilde{\mathcal{X}} | \mathcal{X})\approx G$. Then there is a natural group homomorphism $G \to \mathrm{Aut}(A)$ which can be regarded as covariant functor $G \to \mathscr{C}$. If $\mathcal{X}$ is locally compact and Hausdorff than from \cite{ivankov:infinite_cov_pr} it follows that there is a noncommutative covering projection $\left(C_0(\mathcal{X}), C_0(\mathcal{\widetilde{X}}), \ _{C_0(\mathcal{X})}X_{C_0(\mathcal{\widetilde{X}})} \ , G\right)$. So a noncommutative local system is a generalization of a commutative one. 

\end{exm}

\section{Noncommutative bundles with flat connections}
\subsection{Cotensor products}

\begin{empt}
{\it Cotensor products associated with Hopf algebras}.
Let $H$ be a Hopf algebra over a commutative
ring $k$, with bijective antipode $S$. We use the Sweedler
notation \cite{karaali:ha} for the comultiplication on $H$: $\Delta(h)= h_{(1)}\otimes
h_{(2)}$. $\mathcal{M}^H$ (respectively ${}^H\mathcal{M}$) is the category of
right (respectively left) $H$-comodules. For a right $H$-coaction
$\rho$ (respectively a left $H$-coaction $\lambda$) on a
$k$-module $M$, we denote
$$\rho(m)=m_{[0]}\otimes m_{[1]}\quad \ \mathrm{and} \ \quad\lambda(m)=m_{[-1]}\otimes m_{[0]}.$$
Let $M$ be a right $H$-comodule, and $N$ a left $H$-comodule. The cotensor
product $M\square_H N$ is the $k$-module
\begin{equation}\label{cotensor_hopf}
M\square_H N= \left\{\sum_i m_i\otimes n_i\in M\otimes N~|~\sum_i \rho(m_i)\otimes n_i=
\sum_i m_i\otimes \lambda(n_i)\right\}.
\end{equation}
If $H$ is cocommutative, then $M\square_H N$ is also a right (or left) $H$-comodule.

\end{empt}
\begin{empt}
{\it Cotensor products associated with groups}.
Let $G$ be a finite group. A set $H = \mathrm{Map}(G, \mathbb{C})$ has a natural structure of commutative Hopf algebra
(See \cite{hajac:toknotes}). Addition (resp. multiplication) on $H$ is pointwise addition (resp. pointwise
multiplication). Let $\delta_g\in H, ( g \in G)$ be such that
\begin{equation}\label{group_hopf_action_rel}
\delta_g(g')\left\{
\begin{array}{c l}
    1 & g'=g\\
   0 & g' \ne g
\end{array}\right.
\end{equation}

Comultiplication  $\Delta: H \rightarrow H \otimes H$ is induced by group multiplication
\begin{equation}\nonumber
\Delta f(g) = \sum_{g_1 g_2 = g} f(g_1) \otimes f(g_2); \ \forall f \in \mathrm{Map}(G, \mathbb{C}), \ \forall g\in G.
\end{equation}
i.e.
\begin{equation}\nonumber
\Delta \delta_g = \sum_{g_1 g_2 = g} \delta_{g_1} \otimes \delta_{g_2}; \  \forall g\in G,
\end{equation}

Let $M$ (resp. $N$) be a linear space with right (resp. left) action of $G$ then 

\begin{equation}\label{cotensor_g}
M\square_{\mathrm{Map}(G,\mathbb{C})}N = \left\{\sum_i m_i\otimes n_i\in M\otimes N~|~\sum_i m_i g\otimes n_i=
\sum_i m_i\otimes gn_i;~\forall g\in G\right\}.
\end{equation}

Henceforth we denote by $M\square_GN$ a cotensor product $M\square_{\mathrm{Map}(G,\mathbb{C})}N$.

\end{empt}

\subsection{Bundles with flat connections in differential geometry}\label{fvb_dg}
\paragraph{} I follow to \cite{koba_nomi:fgd} in explanation of the differential geometry and flat bundles.  
%\begin{lem}(Corollary 9.2 \cite{koba_nomi:fgd}) Let $\Gamma$ be a connection in $P(M,G)$ such that the curvature vanishes identically. If $M$ is paracompact and simply connected, then $P$ is isomorphic with the canonical flat connection in $M\times G$.%\end{lem}
\begin{prop}\label{comm_cov_mani}(Proposition 5.9 \cite{koba_nomi:fgd})
\begin{enumerate}
\item Given a connected manifold $M$ there is a unique (unique up to isomorphism) universal covering manifold, which will be denoted by $\widetilde{M}$.
\item The universal covering manifold $\widetilde{M}$ is a principal fibre bundle over $M$ with group $\pi_1(M)$ and projection $p: \widetilde{M} \to M$, where $\pi_1(M)$ is the first homotopy group of $M$.
\item The isomorphism classes of covering spaces over $M$ are in 1:1 correspondence with the conjugate classes of subgroups of $\pi_1(M)$. The correspondence is given as follows. To each subgroup $H$ of $\pi_1(M)$, we associate $E=\widetilde{M}/H$. Then the covering manifold $E$ corresponding to $H$ is a fibre bundle over $M$ with fibre $\pi_1(M)/H$ associated with the principal bundle  $\widetilde{M}(M, \pi_1(M))$. If $H$ is a normal subgroup of $\pi_1(M)$, $E=\widetilde{M}/H$ is a principal fibre bundle with group $\pi_1(M)/H$ and is called a regular covering manifold of $M$.
\end{enumerate}
\end{prop}
\paragraph{}Let $\Gamma$ be a flat connection $P(M, G)$, where $M$ is connected and paracompact. Let $u_0\in P$; $M^*=P(u_0)$, the holonomy bundle through $u_0$; $M^*$ is a principal fibre bundle over $M$ whose structure group is the holomomy group $\Phi(u_0)$. In \cite{koba_nomi:fgd} is explained that $\Phi(u_0)$ is discrete, and since $M^*$ is connected, $M^*$ is a covering space of $M$. Set $x_0=\pi(u_0)\in M$. Every closed curve of $M$ starting from $x_0$ defines, by means of the parallel displacement along it, an element of $\Phi(u_0)$. In \cite{koba_nomi:fgd} it is explained that the same element of the first homotopy group $\pi_1(M, x_0)$ give rise to the same element of $\Phi(u_0)$. Thus we obtain a homomorphism of $\pi_1(M, x_0)$ onto $\Phi(u_0)$. Let $N$ be a normal subgroup of $\Phi(u_0)$ and set $M'=M^*/N$. Then $M'$ is principal fibre bundle over $M$ with structure group $\Phi(u_0)/N$. In particular $M'$ is a covering space of $M$. Let $P'(M',G)$ be the principal fibre bundle induced by covering projection $M'\to M$. There is a natural homomorphism $f: P' \to P$ \cite{koba_nomi:fgd}.
\begin{prop}\label{flat_dg_prop} (Proposition 9.3 \cite{koba_nomi:fgd})
There exists a unique connection $\Gamma'$ in $P'(M',G)$ which is is mapped into $\Gamma$ by homomorphism $f: P'\to P$. The connection $\Gamma'$ is flat. If $u'_0$ is a point of $P'$ such that $f(u'_0)=u_0$, then the holonomy group $\Phi(u'_0)$ of $\Gamma'$ with reference point $u'_0$ is isomorphically mapped onto $N$ by $f$.
\end{prop}
\begin{empt}\label{dg_fl_con_ingr}{\it Construction of flat connections}
Let $M$ be a manifold. Proposition \ref{flat_dg_prop} supplies construction of flat bundle $P(M,G)$ which imply following ingredients:
\begin{enumerate}
\item A covering projection $M'\to M$.
\item A principal bundle $P'(M', G)$ with a flat connection $\Gamma$. 
\end{enumerate}
\end{empt}
\begin{empt}\label{can_fl_conn}{\it Associated vector bundle}.
A principal bundle $P(M,G)$ and a flat connection $\Gamma$ are given by these ingredients. If $G$ acts on $\mathbb{C}^n$ then there is an associated with $P(M,G)$ vector fibre $\mathcal{F}$ bundle with a standard fibre $\mathbb{C}^n$. A space $F$ of continuous sections of $\mathcal{F}$ is a finitely generated projective $C(M)$-module. See \cite{koba_nomi:fgd}.
\end{empt}
\begin{empt}\label{can_fl_bundle}{\it Canonical flat connection and flat bunles}. 
There is a specific case of flat principal bundle such that $P'=M'\times G$ and $\Gamma$ is a canonical flat connection \cite{koba_nomi:fgd}. In this case the existence of $P(M,G)$ depends only on $\pi_1(M)$ and does not depend on differential structure of $M$.
\end{empt}

\begin{empt}\label{comm_fund_k}{\it Local systems and $K$-theory}. 
If $R(G)$ is the group representation ring and $R_0(G)$ is a subgroup of zero virtual dimension then there is a natural homomorphism $R_0(G) \to K^0(M)$ described in \cite{gilkey:odd_space,wolf:const_curv}.
\end{empt}

\subsection{Topological noncommutative bundles with flat connections}
\paragraph{} There are noncommutative  generalizations of described in \ref{fvb_dg} constructions. According to Serre Swan theorem \cite{karoubi:k} any vector bundle over space $\mathcal{X}$ corresponds to a projective $C_0(\mathcal{X})$ module.
\begin{defn}
Let $\left(A, \widetilde{A}, G\right)$  be a  finite noncommutative covering projection. According to definition \ref{loc_sys_defn} any group homomorphism $G \to U(n)$ is a local system. There is a natural linear action of $G$ on $\mathbb{C}^n$, and $\widetilde{A}\square_G\mathbb{C}^n$ is a left $A$-module which is said to be a {\it topological noncommutative bundle with flat connection}.
\end{defn}
\begin{lem}
Let  $\left(A, \widetilde{A}, G\right)$ be a finite noncommutative covering projection, and let $P = \widetilde{A}\square_G\mathbb{C}^n$ be a  topological noncommutative bundle with flat connection. Then $P$ is a finitely generated projective left and right $A$-module.
\end{lem}
\begin{proof}
According to definition $\widetilde{A}$ is a left finitely generated projective $A$-module. A left $A$-module $\widetilde{A}\otimes_{\mathbb{C}}\mathbb{C}^n$ is also finitely generated and projective because $\widetilde{A}\otimes_{\mathbb{C}}\mathbb{C}^n \approx \widetilde{A}^n$. There is a projection $p: \widetilde{A}\otimes_{\mathbb{C}}\mathbb{C}^n \to \widetilde{A}\otimes_{\mathbb{C}}\mathbb{C}^n$ given by:
\begin{equation*}
p(a \otimes x) = \frac{1}{|G|}\sum_{g\in G} ag \otimes g^{-1}x.
\end{equation*}
The image of $p$ is $P$, therefore $P$ is projective left $A$-module. Similarly we can prove that $P$ is a finitely generated projective right $A$-module
\end{proof}

\begin{exm}
Let $M$ be a differentiable manifold $M'\to M$ is a covering projection $P'=M' \times U(n)$ is a principal bundle with a canonical flat connection $\Gamma'$. So there are all ingredients of \ref{dg_fl_con_ingr}. So we have a principal bundle $P(M, U(n))$ with a flat connection $\Gamma$. There is a  noncommutative covering projection $\left(C(M), C(M'), _{C(M')}X_{C(M)}, G\right)$. Let $\mathcal{F}$ (resp.  $\mathcal{F'}$) be a vector bundle associated with $P(M,U(n))$ (resp. $P(M',U(n))$), and let $F$ (resp. $F'$) be a projective finitely generated $C(M)$ (resp. $C(M')$ module which corresponds to $\mathcal{F}$ (resp.  $\mathcal{F'}$). Then we have  $F = C(M')\square_GF'$, i.e. $F$ is a topological flat bundle. 
\end{exm}
\begin{rem}
Since existence of $P(M, U(n)$ depend on topology of $M$ only we use a notion "topological noncommutative bundle with flat connection" is used for its noncommutative generalization.
\end{rem}
\begin{exm}\label{nc_torus_fin_cov}
Let $A_{\theta}$ be a noncommutative torus $\left(A_{\theta}, A_{\theta'}, \mathbb{Z}_m\times\mathbb{Z}_n\right)$ a Galois triple described in \cite{ivankov:infinite_cov_pr}. Any group homomorphism $\mathbb{Z}_m\times\mathbb{Z}_n\to U(1)$ induces a topological noncommutative flat bundle.
\end{exm}
\subsection{General noncommutative bundles with flat connections}
\begin{empt}\label{n_f_b_constr} A vector fibre bundle with a flat connection is not necessary a topological bundle with flat connection, since proposition \ref{fvb_dg} and construction \ref{dg_fl_con_ingr} does not require it. However general case of \ref{fvb_dg} and construction \ref{dg_fl_con_ingr} have a noncommutative analogue. The analogue requires a noncommutative generalization of differentiable manifolds with flat connections. Generalization of a spin manifold is a spectral triple \cite{connes:c_alg_dg,connes:ncg94,varilly:noncom,varilly_bondia}. 
First of all we generalize the proposition \ref{comm_cov_mani}.

Suppose that there is a spectral triple $(\mathcal{B}, H, D)$ such that
\begin{itemize}
\item $\mathcal{B} \subset B$ is a pre-$C^*$-algebra which is a dense subalgebra in $B$.
\item there is a faithful representation $B \to B(H)$.
\end{itemize}
Let $\left(B, A, G\right)$ be a finite noncommutative covering projection.
According to 8.2 of \cite{ivankov:infinite_cov_pr} there is the spectral triple $(\mathcal{A}, A \otimes_BH, \widetilde{D} )$ such that
 \begin{itemize}
\item $\mathcal{A} \subset A$ is a pre-$C^*$-algebra which is a dense subalgebra of $A$.
\item $\widetilde{D}gh = g\widetilde{D}h$, for any $g \in G$, $h \in \Dom \widetilde{D}$.
\end{itemize}

Let $\mathcal{F}$ be a finite projective right $\mathcal{B}$-module with a flat connection $\nabla: \mathcal{F} \to \mathcal{F} \otimes_{\mathcal{B}} \Omega^1(\mathcal{B})$.
Let $\mathcal{E} = \mathcal{F}\otimes_{\mathcal{B}} \mathcal{A}$ be a projective finitely generated $\mathcal{A}$-module and the action of $G$ on $\mathcal{E}$ is induced by the action of $G$ on $\mathcal{A}$. According to \cite{ivankov:infinite_cov_pr} connection $\nabla$ can be naturally lifted to $\widetilde{\nabla}: \mathcal{E}\to \mathcal{E} \otimes_{\mathcal{B}} \Omega^1(\mathcal{B})$.  Let $\mathcal{E}'$ be an isomorphic to $\mathcal{E}$ as $\mathcal{A}$-module and there is an action of $G$ on $\mathcal{E}'$ such that
\begin{equation}\label{twisted_act}
g(xa)=(gx)(ga); \ \forall x \in \mathcal{E}, \ \forall a \in \mathcal{A}, \ \forall g \in G.
\end{equation}
Different actions of $G$ give different $\mathcal{B}$-modules $\mathcal{F} = \mathcal{E}\square_G \mathcal{A}$, $\mathcal{F}' = \mathcal{E}'\square_G \mathcal{A}$. Both $\mathcal{F}$ and $\mathcal{F}'$ can be included into following  sequences
\begin{equation}\label{seqf}
 \mathcal{F} \xrightarrow{i}  \mathcal{E} \xrightarrow{p} \mathcal{F},
\end{equation}
\begin{equation*}
 \mathcal{F}' \xrightarrow{i'}  \mathcal{E}' \xrightarrow{p'} \mathcal{F}'.
\end{equation*}
These sequences induce following
\begin{equation}\label{seqfo}
 \mathcal{F} \otimes_{\mathcal{B}} \Omega^1(\mathcal{B}) \xrightarrow{i \otimes \mathrm{Id}_{\Omega^1(\mathcal{B})}}  \mathcal{E}\otimes_{\mathcal{B}}\Omega^1(\mathcal{B}) \xrightarrow{p \otimes \mathrm{Id}_{\Omega^1(\mathcal{B})}} \mathcal{F}\otimes_{\mathcal{B}}\Omega^1(\mathcal{B}),
\end{equation}
\begin{equation*}
 \mathcal{F}' \otimes_{\mathcal{B}} \Omega^1(\mathcal{B}) \xrightarrow{i' \otimes \mathrm{Id}_{\Omega^1(\mathcal{B})}}  \mathcal{E}'\otimes_{\mathcal{B}}\Omega^1(\mathcal{B}) \xrightarrow{p' \otimes \mathrm{Id}_{\Omega^1(\mathcal{B})}} \mathcal{F}'\otimes_{\mathcal{B}}\Omega^1(\mathcal{B}),
\end{equation*}
The connection $\nabla$ is given by
\begin{equation*}
\nabla p(x) = \left(p \otimes \mathrm{Id}_{\Omega^1(\mathcal{B})}\right)\left(\widetilde{\nabla}(x)\right); \ x \in \mathcal{E}. 
\end{equation*}
From \eqref{seqf} and \eqref{seqfo} it follows that if $y\in \mathcal{F}$ then $\nabla y$ does not depend on $x \in \mathcal{E}$ such that $y=p(x)$. Similarly there is a flat connection  $\nabla': \mathcal{F}' \to \mathcal{F}' \otimes_{\mathcal{B}} \Omega^1(\mathcal{B})$ given by 
\begin{equation*}
\nabla' p'(x) = \left(p' \otimes \mathrm{Id}_{\Omega^1(\mathcal{B})}\right)\left(\widetilde{\nabla'}(x)\right); \ x \in \mathcal{E}'.  
\end{equation*}
Following table explains a correspondence between the proposition \ref{comm_cov_mani} and the above construction.
\newline
\newline
\break
\begin{tabular}{|c|c|}
\hline
DIFFERENTIAL GEOMETRY & SPECTRAL TRIPLES\\
\hline
Manifold $M$ & Spectral triple $(\mathcal{B}, H, D)$\\
The  covering manifold $E$   & Spectral triple $(\mathcal{A}, \mathcal{A} \otimes_{\mathcal{B}}H, \widetilde{D} )$ \\
A regular covering projection $E\to M$  & A noncommutative covering projection $(B, A, G)$ \\
Group of covering transformations $\pi_1(M)/H$  & Group of noncommutative covering transformations $G$ \\
A connection on vector fibre bundle $F\to M$  & An operator $\nabla: \mathcal{F} \to \mathcal{F} \otimes_{\mathcal{B}} \Omega^1(\mathcal{B})$ \\
\hline
\end{tabular}
\newline
\newline
\break
\end{empt}
\begin{exm}
Let $(\mathcal{A}_{\theta}, H, D)$ be a spectral triple associated to a noncommutative torus $A_{\theta}$ generated by unitary elements $u,v\in A_{\theta}$.
Let $\mathcal{F} = \mathcal{A}^4_{\theta}$ be a free module and let $e_1,..., e_4 \in  \mathcal{F}$ be its generators. Let $\nabla: \mathcal{F} \to \mathcal{F} \otimes \Omega^1(\mathcal{A}_{\theta})$ be a connection given by
\begin{equation*}
\nabla e_1 = c_u e_2 \otimes du, \ \nabla e_2 = -c_u e_1 \otimes du, \ \nabla e_3 = c_v e_4 \otimes dv, \ \nabla e_4 = -c_v e_3 \otimes dv.
\end{equation*}
where $c_u, c_v \in \mathbb{R}$. According to \cite{ivankov:nc_wilson_lines} the connection $\nabla$ is flat. Let $\left(A_{\theta}, A_{\theta'}, \mathbb{Z}_m\times\mathbb{Z}_n\right)$ a Galois triple from example \ref{nc_torus_fin_cov}. This data induces a spectral triple $\left(\mathcal{A}_{\theta'}, \mathcal{A}_{\theta'}\otimes_{\mathcal{A}_{\theta}}H, D\right)$. If $\mathcal{E} = \mathcal{F} \otimes_{\mathcal{A}_{\theta}}\mathcal{A}_{\theta'}$ then
\begin{equation}
\mathcal{E} \approx \mathcal{A}_{\theta'}\otimes \mathbb{C}^{4} \approx \mathcal{A}_{\theta}\otimes \mathbb{C}^{4nm} 
\end{equation}
and there is a natural connection $\widetilde{\nabla}: \mathcal{E} \to \mathcal{E}\otimes_{\mathcal{A}_{\theta}}\Omega^1\left(\mathcal{A}_{\theta}\right)$.
Let $\rho : \mathbb{Z}_m\times\mathbb{Z}_n \to U(4)$ be a nontrivial representation. There is an action of  $\mathbb{Z}_m\times\mathbb{Z}_n$ on  $\mathcal{E}' = \mathcal{A}_{\theta'}\otimes \mathbb{C}^{4}$ given by
\begin{equation*}
g (a \otimes x) = ga \otimes \rho(g)x; \ a \in \mathcal{A}_{\theta}, \ x \in \mathbb{C}^4.
\end{equation*}
which satisfies \eqref{twisted_act}. Then $\mathcal{F}' =  \mathcal{A}_{\theta'}\square_{\mathbb{Z}_m\times\mathbb{Z}_n}\mathcal{E}'$ is a finitely generated $A_{\theta}$ module with a connection $\nabla': \mathcal{F}' \to \mathcal{F} \otimes \Omega^1(\mathcal{A}_{\theta})$ given by the construction \ref{n_f_b_constr}.
\end{exm}

\subsection{Noncommutative bundles with flat connections and $K$-theory}
\paragraph{}A homomorphism $R_0(G) \to K^0(M)$ from \ref{comm_fund_k} can be generalized. Let $\left(A, \widetilde{A}, G\right)$ be a finite noncommutative covering projection and $\rho: G \to U(n)$ is a representation, $\mathrm{triv}_n: G \to U(n)$ is the trivial representation. Suppose that an action if $G$ on $\mathbb{C}^n$ is given by $\rho$. Then a homomorphism $R_0(G) \to K(A)$ is given by
\begin{equation*}
[\rho] - \left[\mathrm{triv}_n\right] \mapsto \left[\widetilde{A}\square_G\mathbb{C}^n\right] - \left[A^n\right].
\end{equation*}

\section{Noncommutative generalization of Borel construction}

\begin{empt}
There is a noncommutative generalization of the Borel construction \ref{borel_const_comm}
\end{empt}
\begin{defn}\label{borel_const_ncomm}
Let $A$, $B$ be $C^*$-algebras, let $G$ be a group which acts on both $A$ and $B$. Let $A\otimes_{\mathbb{C}} B$ is any tensor product such that $A\otimes_{\mathbb{C}} B$ is a $C^*$-algebra. The norm closure of generated by
\begin{equation*}
C = \left\{\sum_i a_i\otimes b_i\in A\otimes_{\mathbb{C}} B~|~\sum_i a_i g\otimes b_i=
\sum_i a_i\otimes gb_i;~\forall g\in G\right\}.
\end{equation*}

subalgebra  is said to be a {\it cotensor product of $C^*$-algebras}. Denote by $A \square_G B$ the cotensor product.
\end{defn}
\begin{rem}
We do not fix a type of a tensor product because different applications can use different tensor products (See \cite{bruckler:tensor}).
\end{rem}

\begin{exm}
Let $\mathcal{X}$, $\mathcal{Y}$ be locally compact Hausdorff spaces and let $G$ be a finite or countable group which acts on both $\mathcal{X}$ and $\mathcal{Y}$. Suppose that action on $\mathcal{X}$ (resp. $\mathcal{Y}$) is right (resp. left).
Then there is natural right (resp. left) action of on $C_0(\mathcal{X})$ (resp. $C_0(\mathcal{Y})$). From \cite{bruckler:tensor} it follows that the minimal and the maximal norm on $C_0(\mathcal{X}) \otimes_{\mathbb{C}} C_0(\mathcal{Y})$ coincide. It is well known that $C_0(\mathcal{X}\times\mathcal{Y}) \approx C_0(\mathcal{X}) \otimes_{\mathbb{C}} C_0(\mathcal{Y})$. Let $\mathcal{Z} = \mathcal{X}\times\mathcal{Y} / \approx$ where $\approx$ is given by
\begin{equation*}
(xg, y) \approx (x, g^{-1}y).
\end{equation*}
It is clear that $C_0(\mathcal{Z}) \approx C_0(\mathcal{X}) \square_G C_0(\mathcal{Y})$. 
\end{exm}

\begin{defn}
Let $\left(A, \widetilde{A}, _{\widetilde{A}}X_A, G\right)$ be a Galois quadruple such that there is right action of $G$ of $\widetilde{A}$ and left action of $G$ on $C^*$-algebra $B$. A cotensor product $\widetilde{A}\square_G B$ is said to be a {\it noncommutative Borel construction}.
\end{defn}

\begin{exm}
Let $p: \widetilde{\mathcal{B}}\to \mathcal{B}$ be a topological normal covering projection of locally compact topological spaces, and $G = G(\widetilde{\mathcal{B}}| \mathcal{B})$ is a group of covering transformations. Then $p$ is a principal $G(\widetilde{\mathcal{B}}| \mathcal{B})$-bundle. Let $\mathcal{F}$ be a locally compact topological space with action of $G$ on it. Then there is a natural isomorphism with the $C^*$-algebra of a topological Borel construction
\begin{equation*}
C_0(\widetilde{\mathcal{B}}\times_G\mathcal{F})\approx C_0(\widetilde{\mathcal{B}})\square_GC_0(\mathcal{F}).
\end{equation*}
\end{exm}

%For a fibre bundle (or fibration) $F\to E\to B$ the long exact homotopy sequence $\cdots \to \pi_2(B)\to \pi_1(F)\to\pi_1(E) \to \pi_1(B)\to \pi_0(F) \to \cdots$ gives some information on the fundamental group of the total space: It is a group extension of a subgroup of $\pi_1(B)$ by a quotient of $\pi_1(F)$. (In the example of the Hopf fibration mentioned above this quotient is trivial.)


\begin{thebibliography}{10}

%\bibitem{ant_azz_scan:flat_k}Paolo Antonini, Sara Azzali, Georges Skandalis {\it Flat bundles, von Neumann algebras and $K$-theory with $\mathbb{R}/\mathbb{Z}$-coefficients}, arXiv:1308.0218, 2013.

\bibitem{arveson:c_alg_invt} W. Arveson. {\it An Invitation to $C^*$-Algebras}, Springer-Verlag. ISBN 0-387-90176-0, 1981.

%\bibitem{bezandry_diagana:bound_unbound}Paul H. Bezandry, Toka Diagana {\it Bounded and Unbounded Linear Operators}, in {\it Almost Periodic Stochastic Processes}, Springer, 2011.


%\bibitem{blackadar:ko} B. Blackadar. {\it K-theory for Operator Algebras}, Second edition. Cambridge University Press 1998.

%\bibitem{blackadar:shape_theory} B. Blackadar, {\it Shape theory for $C^*$-algebras}, Math. Scand. 56 , 249-275, 1985.

%\bibitem{blecher:hilb_gen} D.P. Blecher. {\it A generalization of Hilbert modules}, J.Funct. An. 136, 365-421 1996.

%\bibitem{bourbaki_sp:gt} N. Bourbaki, {\it General Topology}, Chapters 1-4, Springer, Sep 18, 1998.

\bibitem{bruckler:tensor} Franka Miriam Br\"uckler. {\it Tensor products of $C^*$-algebras, operator spaces and Hilbert $C^*$-modules}. Mathematical Communications 4(1999), 1999.


\bibitem{morita_hopf_galois}
S. Caenepeel, S. Crivei, A. Marcus, M. Takeuchi. {\it Morita equivalences induced by bimodules over Hopf-Galois extensions.}
arXiv:math/0608572, 2007.

\bibitem{connes:c_alg_dg} Alain Connes. {\it $C^*$-algebras and differential geometry}. arXiv:hep-th/0101093, 2001.

\bibitem{connes:ncg94} Alain Connes. {\it Noncommutative Geometry},
Academic Press, San Diego, CA,  661 p., ISBN 0-12-185860-X, 1994.

%\bibitem{connes_landi:isospectral} Alain Connes, Giovanni Landi. {\it Noncommutative Manifolds the Instanton Algebra and Isospectral Deformations}, arXiv:math/0011194, 2001.

%\bibitem{connes_marcolli:motives}
%Alain Connes, Matilde Marcolli. {\it Noncommutative Geometry, Quantum Fields and Motives},  American Mathematical Society, Colloquium Publications, 2008.

% \bibitem{connes_moscovici:local_index} A. Connes and H. Moscovici, {\it The local index theorem in noncommutative geometry"}. Geom. and Funct. Anal., 1996.


%\bibitem{cuntz_quillen:alg_ext} Joachim Cuntz, Daniel Quillen.  {\it Algebra extensions and nonsingularity}, J. Amer. Math. Soc. 8 251-289, 1995

\bibitem{davis_kirk_at}
James F. Davis. Paul Kirk
{\it Lecture Notes in Algebraic Topology}. Department of Mathematics, Indiana University, Blooming- ton, IN 47405, 2001.

%\bibitem{dixmier_tr}J.Dixmier. {\it Traces sur les $C^*$-algebras}. Ann. Inst. Fourier, 13, 1(1963), 219-262, 1963.

%\bibitem{moyal_spectral} V. Gayral, J. M. Gracia-Bond\'{i}a, B. Iochum, T. Sch\"{u}cker, J. C. Varilly {\it Moyal Planes are Spectral Triples}. arXiv:hep-th/0307241, 2003

\bibitem{gilkey:odd_space}P.B. Gilkey. {\it The eta invariant and the $K$-theory of odd dimensional spherical space forms}.Inventiones mathematicae, Springer-Verlag, 1984.


%\bibitem{nicolas_ginoux:dirac_spectrum}Nicolas Ginoux. {\it The Dirac Spectrum.}Springer, Jun 11, 2009.

\bibitem{varilly_bondia}
Jos\'e M. Gracia-Bondia, Joseph C. Varilly, Hector Figueroa, {\it Elements of Noncommutative Geometry}, Springer, 2001.

\bibitem{gross_gauge}David J. Gross. {\it
Gauge Theory-Past, Present, and Future?} Joseph Henry Luborutoties, Ainceton University, Princeton, NJ 08544, USA. (Received November 3,1992).


\bibitem{halmos:set} Paul R.  Halmos {\it Naive Set Theory.} D. Van Nostrand Company, Inc., Prineston, N.J., 1960.


%\bibitem{helemsky:qfa} A. Ya. Helemsky. {\it Quantum Functional Analysis. Non-Coordinate Approach.} Providence, R.I. : American Mathematical Society, 2010.

\bibitem{ivankov:infinite_cov_pr} Petr Ivankov.
 {\it Infinite Noncommutative Covering Projections}.  arXiv:1405.1859, 2014.

 \bibitem{ivankov:nc_wilson_lines} Petr Ivankov.
  {\it Noncommutative Generalization of Wilson Lines}. arXiv:1408.4101, 2014.

\bibitem{hajac:toknotes}
{\it Lecture notes on noncommutative geometry and quantum groups}, Edited by Piotr M. Hajac.


%\bibitem{ivankov:nc_cov_k_hom}Petr Ivankov. {\it Noncommutative covering projections and $K$-homology}, 	arXiv:1402.0775, 2014.

%\bibitem{ivankov:uni_nc_cov}
%Petr R. Ivankov. {\it Universal covering space of the noncommutative torus},
%arXiv:1401.6748, 2014.

\bibitem{kakariadis:corr}
Evgenios T.A. Kakariadis, Elias G. Katsoulis, {\it Operator algebras and $C^*$-correspondences: A survey.} 	arXiv:1210.6067, 2012.

\bibitem{karaali:ha} Gizem Karaali {\it On Hopf Algebras and Their Generalizations}, arXiv:math/0703441, 2007.



\bibitem{karoubi:k} M. Karoubi. {\it K-theory, An Introduction.} Springer-Verlag 1978.

\bibitem{koba_nomi:fgd} S. Kobayashi, K. Nomizu. {\it Foundations of Differential Geometry}. Volume 1. Interscience publishers a division of John Willey \& Sons, New York - London. 1963.

%\bibitem{bram:atricle}
%Bram Mesland. {\it Unbounded bivariant $K$-theory and correspondences in noncommutative geometry} arXiv:0904.4383, 2009


\bibitem{miyashita_fin_outer_gal} Y\^oichi Miyashita, {\it Finite outer Galois theory of noncommutative rings}. Department of Mathematics, Hokkaido, University, 1966.

\bibitem{miyashita_infin_outer_gal} Y\^oichi Miyashita, {\it Locally finite outer Galois theory}. Department of Mathematics, Hokkaido, University, 1967.

%\bibitem{muhly_williams:groupoid_ctr}  Paul S. Muhly and Dana P. Williams. {\it Continuous trace groupoid $C^*$-algebras.}, Math. Scand. 1990

%\bibitem{munkres:topology} James R. Munkres. {\it Topology.} Prentice Hall, Incorporated, 2000.

\bibitem{pedersen:ca_aut}Pedersen G.K. {\it $C^*$-algebras and their automorphism groups}. London ; New York : Academic Press, 1979.

%\bibitem{phillips:inv_lim_app} N. Christopher Phillips {\it Inverse Limits of $C^*$ - algebras and Applications.} University of California at Los Angeles, Los Angeles, CA 90024, 1991

%\bibitem{switzer:at} Switzer R M, {\it Algebraic Topology - Homotopy and Homology},
%Springer. 2002



%\bibitem{adams:infinite_loop_spaces} J. F. Adams. {\it Infinite loop spaces}. Ann. of Math. Studies no. 90, Princeton Univ. Press, Princeton, N. J., 1978

%\bibitem{phillips:c_infty_loop} N. Christopher Philllips. {\it $C^{\infty}$ Loop Algebras and Noncommutative Bott Periodicity}. Transactions of the American Matematical Society, Volume 325, Number 2, June 1991

%\bibitem{sitarz:equiv} Andrzej Sitarz {\it Equivariant spectral triples}, Noncommutative Geometry and Quantum Groups (Piotr M. Hajac and Wieslaw Pusz, eds.), Banach Center Publ., vol 61, Polish Acad. Sci., pp. 231-268,  Warsaw 2003



%\bibitem{cuntz:o_n} J. Cuntz, {\it Simple $C^*$ - algebras generated by isometries}, Comm. Math. Phys. 57:2, 1977

%\bibitem{cuntz:k_o_n} J. Cuntz, {\it$K$ - theory of certain $C^*$ - algebras}, Ann. of Math. (2), 113:1 1981


%\bibitem{Cohn:68} Paul~Moritz Cohn. {\it {M}orita equivalence and duality}, Queen Mary College   Mathematics Notes, Dillon's Q.M.C.\ Bookshop, London, 1968.


%\bibitem{bourbaki_sp:gt} N. Bourbaki, {\it General Topology}. Chapters 1-4, Springer, Sep 18, 1998

%\bibitem{williams_sp:morita_cont_trace_alg} Iain Raeburn, Dana P. Williams. {\it Morita Equivalence and Continuous-Trace $C^*$-Algebras}. American Mathematical Soc., 1998

%\bibitem{dixmier_tr}J.Dixmier. {\it Traces sur les $C^*$-algebras}. Ann. Inst. Fourier, 13, 1(1963), 219-262, 1963

%\bibitem{baum_higson_schik:kh}Paul Baum, Nigel Higson, and Thomas Schick. {\it On the Equivalence of Geometric and Analytic $K$-Homology}. Pure and Applied Mathematics Quarterly Volume 3, Number 1 (Special Issue: In honor of Robert MacPherson, Part 3 of 3) 1-24, 2007

%\bibitem{meyer:morita} Ralf Meyer. {\it Morita Equivalence In Algebra And Geometry.} math.berkeley.edu/~alanw/277papers/meyer.tex, 1997



%\bibitem{rumynin_hopf_galois_ci} Dmitriy Rumynin  {\it Hopf-Galois extensions with central invariants.}  arXiv:q-alg/9707021 1997

%\bibitem{dixmier_a_r} Jacques Dixmier {\it Les C*-alg\`{e}bres et leurs repr\'esentations} 2e \'ed. Gauthier-Villars in Paris 1969



\bibitem{rieffel_finite_g} Marc A. Reiffel, {\it Actions of Finite Groups on $C^*$ - Algebras}. 	Department of Mathematics University of California Berkeley. Cal. 94720 U.S.A. 1980.

\bibitem{rieffel_morita}Marc A. Reiffel, {\it Morita equivalence for $C^*$-algebras and $W^*$-algebras }, Journal of Pure and Applied Algebra 5 (1974), 51-96. 1974.

%\bibitem{dixmier_douady_d} Claude Schochet, {\it Dixmier-Douady for Dummies}. 	arXiv:0902.2025 2009.



\bibitem{spanier:at}
E.H. Spanier. {\it Algebraic Topology.} McGraw-Hill. New York 1966.


\bibitem{takeuchi:inf_out_cov}Takeuchi, Yasuji {\it Infinite outer Galois theory of non commutative rings} Osaka J. Math.
Volume 3, Number 2, 1966.

%\bibitem{inikolaev:c_bundles} Igor Nikolaev, {\it Topology of the $C^*$ algebra bundles}. Centre interuniversitaire de recherche en g\'eom\'etrie diff\'erentielle et topologie UQAM Montr\'eal H3C 3P8 Canada 1999.




%\bibitem{bass} H. Bass. {\it Algebraic K-theory.} W.A. Benjamin, Inc. 1968.

%\bibitem{thompsen:homtop}
%Klaus Thompsen. {\it Homotopy classes of * - homomorphisms between stable $C^*$ - algebras and their muliplier algebras.} Duke Matematical Journal (C) August 1990.




%\bibitem{blackadar:oa}
%B. Blackadar. {\it Operator Algebras Theory of $C^*$ Algebras and von Neumann Algebras}. Springer-Verlag Berlin Heidelberg 2006


%\bibitem{milne:etale}
%J.S. Milne. {\it Etale cohomology.} Princeton Univ. Press  1980.

%\bibitem{murre:fund}
%J.P. Murre. {\it Lectures on An Introduction to Grothendieck's  Theory of the Fundamental Group.} Notes by S. Anantharaman, Tata Institute of Fundamental Research, Bombay, 1967.

%\bibitem{murphy}
%G.J. Murphy. {\it $C^*$-Algebras and Operator Theory.}Academic Press 1990.


%\bibitem{connes_marcolli::motives} Alain Connes Matilde Marcolli. {\it Noncommutative Geometry, Quantum Fields and Motives.} Preliminatry version. www.alainconnes.org/docs/bookwebfinal.pdf

%\bibitem{mesland::unbounded_biviariant} Bram Mesland. {\it Unbounded biviariant $K$-theory and correspondences in noncommutative geometry}. arXiv:0904.4383. 2009.



%\bibitem{connes:ng} A. Connes. {\it Noncommutative Geometry.} Academic Press, London, 1994.

%\bibitem{brown_green_rieffel:morita_stable} Lawrence G. Brown, Philip Green, and Marc A. Rieffel.  {\it Stable isomorphism and strong Morita equivalence of $C^∗$  -algebras.} Source: Pacific J. Math. Volume 71, Number 2 , 349-363, 1977


%\bibitem{faith:I} C. Faith. Algebra: {\it Rings, Modules and Cathegories I}. Springer-Verlag 1973

%\bibitem{ros_scho:kt_uct} Jonathan Rosenberg, Claude Schochet, {\it The K\"unneth theorem and the universal coefficient theorem for Kasparov's generalized K -functor}, Duke Math. J. Volume 55, Number 2 1987


\bibitem{varilly:noncom}
J.C. V\'arilly. {\it An Introduction to Noncommutative Geometry}. EMS 2006.

%\bibitem{weil:basic_number_theory}Andre Weil {\it Basic Number Theory}. Springer 1995

%\bibitem{Partha_quantum_su} Partha Sarathi Chakraborty, Arupkumar Pal. {\it Equivariant spectral triples on the quantum $SU(2)$ group.} arXiv:math.KT/0201004, 2003.

%\bibitem{geom_anal_k_homology} Paul Baum, Nigel Higson, and Thomas Schick {\it On the Equivalence of Geometric and Analytic K-Homology} arXiv:math/0701484, 2009.

%\bibitem{blackadar:kocalg_neumann} B. Blackadar {\it Operator Algebras Theory of C* - Algebras and von Neumann Algebras}. Springer-Verlag Berlin Heidelberg 2006.





%\bibitem{connesdebois:3dsphere}
%A. Connes, M. Dubois-Violette. Moduli space and structure of
%noncommutative 3-spheres.  LPT-ORSAY 03-34 ; IHES/M/03/56.
%Lett.Math.Phys. 66 91-121. 2003.

%\bibitem{conneslandi:isospectal}
%A. Connes, G. Landi. Noncommutative Manifolds the Instanton Algebra
%and Isospectral Deformations, math.QA/0011194, 2000.

%\bibitem{suprsym:qt}
%J. Fr\"ohlich, O. Grandjean, A. Recknagel. Supersymmetric Quantum
%Theory and (Non-Commutative) Differential Geometry, ETH-TH/96-45
%1996.

%\bibitem{ivankov:fund}
%P.R. Ivankov, N.P. Ivankov. The Noncommutative Geometry
%Generalization of Fundamental Group. arXiv:math.KT/0604508, 2006

%\bibitem{johnstone:topos}
%P.T. Johnstone. Topos Theory, L. M. S. Monographs no. 10, Academic
%Press 1977.

%\bibitem{lang}
%S. Lang. Algebra. Addison-Wesley Publishing Company, Reading, Mass
%1965.

%\bibitem{reconstr}
%A. Rennie, J.C. V\'arilly. Reconstruction of Manifolds in
%Noncommutative Geomery. \newline arXiv:math/0610418v3 [math.OA] 24
%Mar 2007.


%\bibitem{varilly:lecture}
%J.C. V\'arilly. Dirac operators and Spectral Geometry. Lecture notes
%by Pave{\l} Witkowsky from Warshaw Noncommutative Geometry, January
%2006.%http://ncg.mimuw.edu.pl/index.php?option=com_docman&task=doc_download&gid=10&Itemid=58

\bibitem{wolf:const_curv}
Wolf, J. {\it Spaces of constant curvature}. New York: McGraw-Hill, 1967.

%\bibitem{raimar_wulkenhaar:nc_spectral_triple}
%Raimar Wulkenhaar.{\it Non-compact spectral triples with finite volume}. 	arXiv:0907.1351, 2009.




\end{thebibliography}
\end{document}